\documentclass[12pt]{amsart}
\usepackage{amsmath, amssymb, amsfonts, amsthm, graphicx}
\newcommand{\nc}{\newcommand}
\nc{\x}{\tau_{0}}
\nc{\y}{\tau_{2}}
\nc{\z}{\tau_{1}}
\nc{\Q}{{\mathcal Q}}

\newtheorem{theorem}{Theorem}
\newtheorem{proposition}{Proposition}
\newtheorem{rmk}{Remark}

\begin{document}

\title{Three-period orbits in billiards on the surfaces of constant curvature }

\author{Victoria Blumen, Ki Yeun Kim, Joe Nance, Vadim Zharnitsky }

\address{Department of Mathematics, University of Illinois, Urbana, IL 61801}

\begin{abstract}
An approach due to  Wojtkovski  \cite{wojtk}, based on the  Jacobi fields, 
is applied to 
study sets of 3-period orbits in billiards on hyperbolic plane and on 
two-dimensional sphere. It is found that the set of 3-period orbits in 
billiards on hyperbolic plane, as in the planar case, has zero measure. 
For the sphere, a new proof of Baryshnikov's theorem is obtained which states 
that 3-period orbits can form a set of positive measure if and only if a certain  
a natural condition on the orbit length is satisfied.  
\end{abstract}

\maketitle

\section{Introduction}
This article provides a unified approach, based on the Jacobi fields, 
to study open sets of 3-period orbits in billiards on manifolds with 
constant curvature. Specifically, we consider spherical and hyperbolic cases. 
While the spherical as well as Euclidean case has been treated previously,
our result for  the billiards on hyperbolic plane is a new one. Billiards on
 manifolds of constant curvature have been studied earlier in \cite{gutkin}, 
\cite{gutkin_t}.
In \cite{gutkin_t}, billiard domains on the sphere, containing open sets 
of periodic orbits, have been constructed. In \cite{ymb}
and in this article, the main goal is to assure that except for those special cases, 3-period orbits have zero measure.

The billiard system on a two dimensional Riemannian manifold $(M,g)$ consists of
 the domain ${\Q}$ with a piecewise smooth boundary $\partial {\Q}$ and
a mass point moving along the geodesics inside the domain. Whenever the mass hits the boundary,
it reflects according to  Fermat's principle so as to extremize the path length. 
That leads to the familiar law:
the angle of incidence  is equal to the angle of reflection.

Periodic orbits are a natural object of study in dynamical systems. One important question concerns
the presence of large sets, in particular sets of positive measure, of periodic orbits in 
the billiard ball problem. Informally speaking,
measure corresponds to the probability that a given orbit is periodic.
This question has been originally motivated by spectral geometry problems. 
The second term of the Weyl asymptotics for the Dirichlet problem in
 a bounded domain has a certain special form if periodic orbits of the associated 
billiard problem have zero measure \cite{ivrii}, see also a survey article 
by Gutkin \cite{gutkin_s}.
There is a natural invariant measure for the billiard map which can be defined as follows:
 let $s$ be 
an arclength parameter coordinatizing the boundary and let $\phi \in [0,\pi]$ be the angle 
of the outcoming ray from the boundary measured in the counterclockwise direction. The billiard
 ball map $T: [\partial \Q\times [0,\pi]] \rightarrow [\partial \Q\times [0,\pi]]$ which takes an
 outcoming ray to another one obtained after reflection from the boundary,
 preserves the measure $\mu = \sin \phi\, d\phi \, ds$, see {\em e.g.} \cite{birkhoff}.

Our  motivation to study the structure of the set of periodic orbits in non-Euclidean geometries
is that this understanding may help one with the planar case for higher period orbits. 
It is also expected that eigenvalue asymptotics  in non-Euclidean geometries would also 
require understanding the structure of the sets of periodic orbits.

For the planar billiard problem, it is easy to see that two period orbits have zero measure,
since these orbits must be normal to the boundary at both ends. Similarly, this is the case for
 a billiard on ${\mathbb H^2}$.
On the other hand, a billiard on $\mathbb{S}^2$ with boundary given by equator has
 two-period orbits of positive measure. This has to do with the presence of conjugated
 points on $\mathbb{S}^2$.

For the period 3, the problem on existence of positive measure sets is already non-trivial.
The first result on zero measure of 3-period orbits in planar billiards was obtained by 
Rychlik, see \cite{rychlik}, relying on symbolic calculations,
 which were later removed in \cite{stojanov}.
Using Jacobi fields, Wojtkovski gave an elegant simple proof of Rychlik's theorem. 
 Subsequently, there have been
extensions to other types of billiard systems: higher dimensional (\cite{vorobets}),
outer billiards (\cite{genin,tumanov}), and spherical (\cite{ymb}). Recently, a proof 
for the period 4 case has been announced by Glutsyuk and Kudryashov \cite{glku}.

Our main result is

\begin{theorem}\label{theo:main}
The set of 3-period orbits in any billiard on ${\mathbb H^2}$ has zero measure.
\end{theorem}

In order to prove this theorem, we extend the Jacobi fields approach
 from  \cite{wojtk} and present the unified proof   
which treats all three billiard systems on the constant curvature manifolds 
in the same manner.
Our argument proceeds independently of the underlying geometry until we get
the compatibility condition. Then, using the relevant cosine formula,
which depends on the geometry, we obtain the relation that must be satisfied on
a neighborhood filled with 3-period orbits
\[
k_g(s_0) =  \sin^3 (\phi_0)F(L),
\]
where $L$ is the length of 3-period orbits, $k_g(s_0)$ is geodesic curvature at 
one of the vertices, $\phi_0$ is the angle of the billiard orbit with the tangent to the boundary  
at this vertex and $s_0$ is the value of the arclength parameter $s$ at the vertex.
The function $F(L)$ depends on the underlying Riemannian manifold
\[
 F(L) =
  \begin{cases}
   \frac{2}{L} & \text{on }  {\mathbb E^2} \\
   \coth\left(\frac{L}{2}\right)    & \text{on } {\mathbb H^2}  \\
   \cot\left(\frac{L}{2}\right)    & \text{on }  {\mathbb S^2}.
  \end{cases}
\]

From this formula it is possible to classify sets of 3-period orbits. 
In particular, we obtain a new proof of a theorem by Baryshnikov on 
the spherical case \cite{ymb} where sub-Riemannian geometry methods were used.

\begin{theorem}\label{theo:sphere}
Let   $P_3$ be the set of 3-period orbits in the billiard domain $\partial \Q$ on ${\mathbb S}^2$.
Assume that some orbit $(x_0,x_1,x_2) \in P_3$ has perimeter $L=\pi,3\pi$ or
 $5\pi$ and that some arcs of $\partial \Q$ containing
$x_0,x_1,x_2$ belong to great circles. Then $(x_0,x_1,x_2) \in int(P_3)$ and $P_3$ has positive measure. Otherwise,   $(x_0,x_1,x_2) \notin int(P_3)$. 
In particular, if none of 3-period orbits are of the above special type, then $P_3$ has an empty 
interior and is the set of zero measure. 
\end{theorem}

\begin{rmk}
We only prove that the set of 3-period orbits on ${\mathbb H}^2$ (and on ${\mathbb S}^2$ when
the special condition is not satisfied) has an empty interior. The stronger statement about zero
 measure follows verbatim  the argument in \cite{wojtk}, page 163. 
\end{rmk}

\section{Billiard system on the surface of constant curvature}
\subsection{Jacobi fields}
Let $\Q$ be a smooth domain on a surface of constant curvature $\varkappa$. The billiard ball inside $\Q$ travels along the geodesics  and reflects at the boundary. Let  $\boldsymbol{\gamma}(\epsilon, \tau)$ be a one-parameter family of geodesics where $|\epsilon| < \epsilon_0, -\infty < \tau < \infty$. 

For the reader's convenience, we briefly recall the derivation of 
the Jacobi fields, see {\em e.g.} \cite{carmo} or \cite{bar}. 
The Jacobi field is defined by
\[
\displaystyle \mathbf{J(\tau)}=\frac{\partial \boldsymbol{\gamma}(0, \tau)}{\partial \epsilon}
\]
and it satisfies the Jacobi equation
\[\displaystyle \frac{\nabla}{d\tau}\frac{\nabla}{d\tau}\mathbf{J(\tau)}+ R(\mathbf{J(\tau)},\dot{\boldsymbol{\gamma}})\dot{\boldsymbol{\gamma}} =0,\]
where $\nabla$ denotes the covariant derivative and $R$ is the curvature tensor.
As usual, we are interested in the component of the Jacobi field, that is perpendicular to $\dot{\boldsymbol{\gamma}}$. Therefore, it can be expressed as
\[
\mathbf{J(\tau)}= J(\tau) \mathbf{n}(\tau),
\]
where $J(\tau)$ is a scalar function and $\mathbf{n}(\tau)$ is a unit vector field perpendicular to $\dot{\boldsymbol\gamma}$. If the surface has constant curvature $\varkappa$, then one obtains 
a scalar equation with constant coefficients
\begin{equation}\label{eq:jcb}
\displaystyle J''(\tau)+ \varkappa J(\tau) =0.
\end{equation}
According to the standard result in the theory of differential equations, the solution of 
the Jacobi equation is uniquely defined if two initial conditions $J(0)$ and  $J'(0)$  are given. 

\subsection{Evolution and reflection matrices}

Consider billiards on the hyperbolic plane $\mathbb{H}^2$ and the 2-sphere $\mathbb{S}^2$ which have the curvature $\varkappa =-1$ and $\varkappa = 1$ respectively.  Solving the Jacobi equation (\ref{eq:jcb}), we obtain
\[
J(\tau)=\left\{ \begin{array}{cl}
J(0) \cosh(\tau)+J'(0) \sinh(\tau)& $on $\mathbb{H}^2 \\
J(0)\cos(\tau) + J'(0)\sin(\tau)& $on $\mathbb{S}^2.\\
 \end{array} \right.
\]

In each case, we obtain the evolution matrix $P(\tau)$
\[
\left( \begin{array}{c}
J(\tau) \\
J'(\tau) \\
 \end{array} \right)
= P(\tau)
\left( \begin{array}{c}
J(0) \\
J'(0) \\
 \end{array} \right),
\]

\begin{equation*}
\textrm{where }P(\tau)=\left\{ \begin{array}{cl}
\left(\begin{array}{rc}
\cosh(\tau) & \sinh(\tau) \\ 
\sinh(\tau) & \cosh(\tau) \end{array}\right)
& \mbox{on }\mathbb{H}^2 \\ \\
\left(\begin{array}{rc}
\,\,\,\cos(\tau) & \sin(\tau)\,\,\,\, \\
\,\,\,\sin(\tau) & \cos(\tau)\,\,\,\, \end{array}\right)
& \mbox{on }\mathbb{S}^2\\
 \end{array} \right.
\end{equation*}
which describes the changes of the Jacobi field over time.

Note that the corresponding evolution matrix in the Euclidean case is given by
\begin{equation*}
P(\tau)= 
\left(\begin{array}{rc}
1 & \tau \\
0 & 1 \end{array}\right).
\end{equation*}

When the billiard ball hits the boundary at $x=(s, \phi)$, the Jacobi field is transformed by the linear map $R(x)$ which is essentially the same as the reflection map in the Euclidean case
\[
\left( \begin{array}{r}
J_{out} \\
J'_{out} \\
 \end{array} \right)
=R(x)
 \left( \begin{array}{r}
J_{in} \\
J'_{in} \\
 \end{array} \right),
\]
\begin{equation*}
\textrm{where }R(x)=\left(\begin{array}{cr}
-1 & 0 \\
\displaystyle \frac{2k_g(s)}{\sin(\phi)} & -1 \end{array}\right).
\end{equation*}
Note that the reflection matrix is directly related to the classical mirror formula, see \cite{wojtk}.

One should expect that the reflection matrix $R(x)$ for the billiard on a two dimensional
Riemannian manifold should have the same form as in the Euclidean case \cite{wojtk}.
Nevertheless, we provide some justification. 
Consider a one-parameter family of geodesics $\boldsymbol \gamma(\epsilon,\tau)$
 reflecting from the billiard boundary $\partial \Q$ on a two dimensional 
Riemannian manifold.
In an $\epsilon$-neighborhood of the reflection point $x_0$ 
of $\boldsymbol\gamma(0,\tau_0)$, the manifold
can be represented as a smooth two dimensional surface in ${\mathbb R}^3$. 
Projecting the geodesics and  the boundary onto the tangent plane at $x_0$,
 we obtain the corresponding structure on the plane: 
a family of orbits reflecting from the boundary.
It is easy to estimate that the angles as well as distances before and after the projection,
differ by $O(\epsilon^2)$. This is mainly due to the expansion 
$\cos \epsilon = 1-\epsilon^2/2+ O(\epsilon^4)$. Also, straightforward estimates show that the projected 
 boundary curve will have the curvature equal to the geodesic curvature of $\partial \Q$ with 
the accuracy $O(\epsilon^2)$. As quadratic terms do not affect linear 
transformations, the reflection map will have the same form as in the Euclidean
case with $k$ replaced by $k_g$.

\section{Billiard on the hyperbolic plane}

\begin{figure}
  \centering
   \includegraphics[trim = 0mm 0mm 200mm 0mm, width=.5\textwidth]{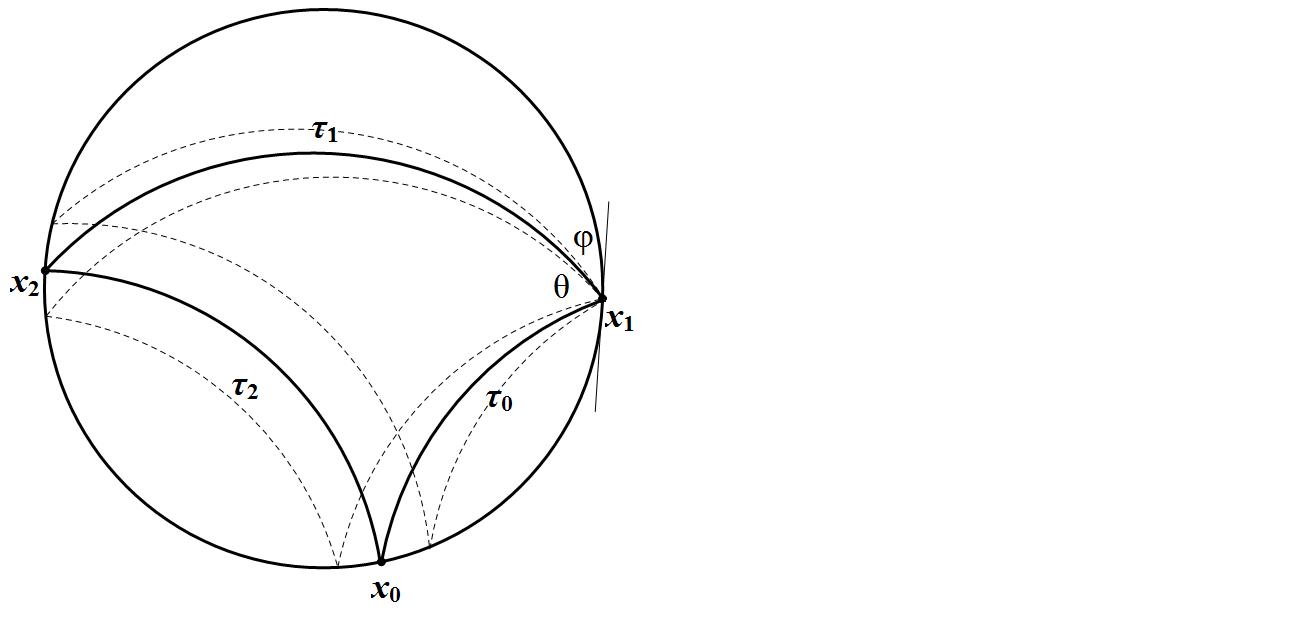}
    \caption{Periodic orbits with period 3 in a billiard on $\mathbb{H}^2$.}
\label{fig:fig1}
 \end{figure}

In this section we prove Theorem \ref{theo:main}. Assume that there is an open set of 3-period orbits. Then we must have $T^3$ and $DT^3$ equal to the identity, which implies
\begin{equation}
 P(\z)R(x_1)P(\x)R(x_0)P(\y)R(x_2) =I,
\end{equation}
where $I$ is the identity map,  $x_0, x_1$ and $x_2$ are the collision points, and $\x,\y,\z$, are 
the distances between collision points, see Figure \ref{fig:fig1}. 
This relation can be also rewritten as
\begin{equation}\label{eq:mtrx}
P(\z)R(x_1)P(\x) = R^{-1}(x_2)P^{-1}(\y)R^{-1}(x_0)
\end{equation}
which takes the form
\[
\begin{bmatrix}  \cosh(\tau_1) & \sinh(\tau_1) \\ \sinh(\tau_1) & \cosh(\tau_1) \end{bmatrix} 
\cdot
\begin{bmatrix} -1 & 0 \\ \frac{2k_g(x_1)}{\sin(\phi_1)} & -1 \end{bmatrix}
\cdot
\begin{bmatrix}  \cosh(\tau_0) & \sinh(\tau_0) \\ \sinh(\tau_0) & \cosh(\tau_0) \end{bmatrix} =
\]
\[
=
\begin{bmatrix} -1 & 0 \\ -\frac{2k_g(x_2)}{\sin(\phi_2)} & -1 \end{bmatrix}
\cdot
\begin{bmatrix}  \cosh(\tau_2) & -\sinh(\tau_2) \\ -\sinh(\tau_2) & \cosh(\tau_2) \end{bmatrix} 
\cdot 
\begin{bmatrix}  -1 & 0 \\ -\frac{2k_g(x_0)}{\sin(\phi_0)} & -1  \end{bmatrix}.
\]

\vspace{3mm}

After simplification, we equate the top right components to get\footnote{Compare with the corresponding formula in the Euclidean case:  $\x+\z-\y = \frac{2 \,k(x_1) \,\,\x \,\z}{\sin(\phi_1)}$,  which was derived in \cite{wojtk}.}, 
\begin{equation}\label{eq:step1}
  \sinh(\x+\z)-\sinh(\y)=\frac{2k_g(x_1)\sinh(\x)\sinh(\z)}{\sin(\phi_1)}.
\end{equation}
We define $\theta$ to be the interior angle between two adjacent segments of an orbit, that is, $\theta=\pi - 2\phi$, 
\mbox{see Figure \ref{fig:fig1}}. 

Then we alter the hyperbolic cosine formula into
\begin{equation*}
\cosh(\y)=\cosh(\x+\z)-\sinh(\x)\sinh(\z)-\sinh(\x)\sinh(\z)\cos(\theta_1).
\end{equation*}
We use the half angle formula to get
\begin{equation}\label{eq:step2}
\cosh(\x+\z)-\cosh(\y)=2 \cos^2\left(\frac{\theta_1}{2}\right)\sinh(\x)\sinh(\z).
\end{equation}
Combining (\ref{eq:step1}) and (\ref{eq:step2}), we arrive at
\begin{equation*}
\frac{\cosh(\x+\z)-\cosh(\y)}{\cos^2\left(\frac{\theta_1}{2}\right)}=\frac{\sinh(\x+\z)-\sinh(\y)\cos\left(\frac{\theta_1}{2}\right)}{k_g(x_1)}.
\end{equation*}

Note that the length of an orbit $L = \x + \z + \y$ is invariant. Therefore, 
\begin{equation*}\begin{split}\\
k_g(x_1)&=\frac{\sinh(\x+\z)-\sinh(\y)}{\cosh(\x+\z)-\cosh(\y)}\cos^3\left(\frac{\theta_1}{2}\right) \\
&=\frac{\sinh(L-\y)-\sinh(\y)}{\cosh(L-\y)-\cosh(\y)}\cos^3\left(\frac{\theta_1}{2}\right) \\
&=\cos^3\left(\frac{\theta_1}{2}\right)\coth\left(\frac{L}{2}\right) \\
&=\sin^3(\phi_1)\coth\left(\frac{L}{2}\right).\end{split}
\end{equation*}

This relation must hold for all nearby orbits. In particular, for all orbits starting at
the same point on the boundary with different angles of reflection. Thus, we obtain a contradiction
 because the right-hand side of the equation is not constant in any interval. 
Therefore, the set of 3-period orbits has an empty interior. Next, following an argument in
\cite{wojtk} we obtain that the set has zero measure, which ends the proof of the Theorem \ref{theo:main}.

\section{Billiard on the 2-sphere}

Now we prove Theorem \ref{theo:sphere} using the same method. Assuming there is an open set of 3-period orbits on $\mathbb{S}^2$, we again obtain that $T^3$ and $DT^3$ are equal to the identity. 
Therefore, using \eqref{eq:mtrx} again
\begin{equation*}
P(\z)R(x_1)P(\x) = R^{-1}(x_2)P^{-1}(\y)R^{-1}(x_0)
\end{equation*}
we get 
\begin{equation}\label{eq:step3}
\sin(\x)\cos(\z)+\cos(\x)\sin(\z)-\sin(\y)=\frac{2k_g(x_1)\sin(\x)\sin(\z)}{\sin(\phi_1)}.
\end{equation}
Note that this relation is the same as \eqref{eq:step1} if trigonometric functions are replaced
 with their hyperbolic counterparts.

Combining (\ref{eq:step3}) and the modified version of spherical cosine formula, we arrive at
\begin{equation}\label{eq:sphrel}
k_g(x_1)=\sin^3(\phi_1)\cot\left(\frac{L}{2}\right).
\end{equation}

\begin{figure}[htd]
    \centering

   \includegraphics[trim = 0mm 0mm 100mm 0mm, width=.7\textwidth]{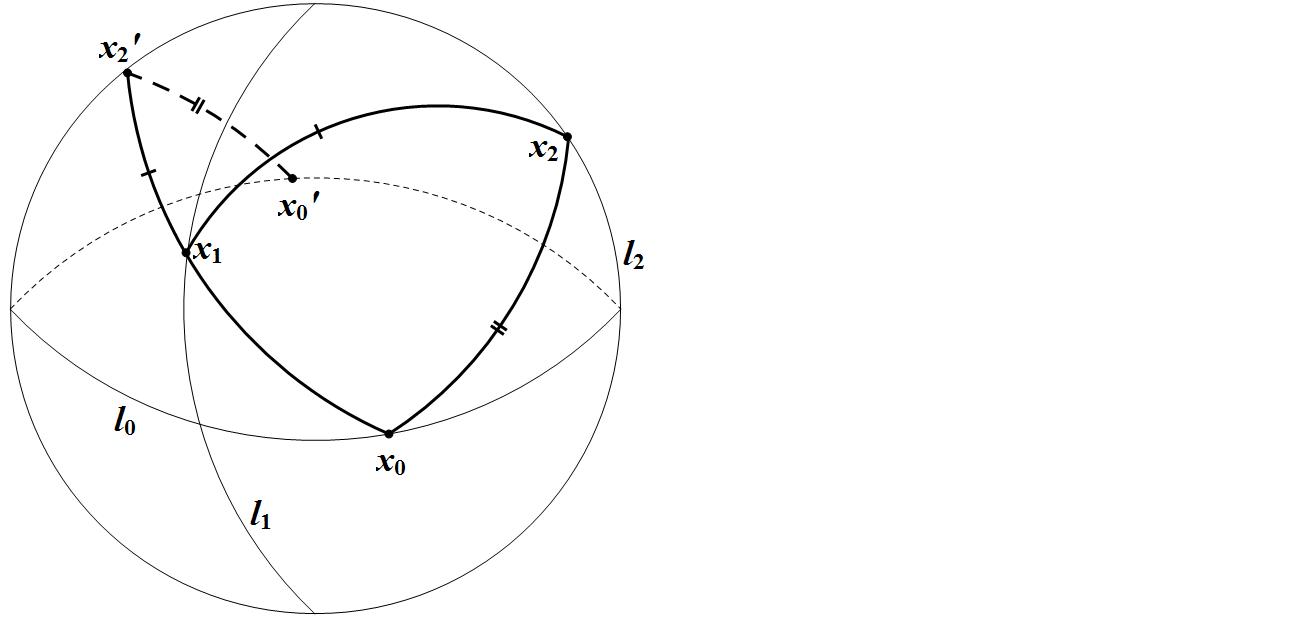}

\vspace{10mm}

    \caption{The special case $L=\pi$. The billiard boundary, which is an octant, is formed by three mutually orthogonal great circles $l_0, l_1, l_2$. The billiard orbit $(x_0,x_1,x_2)$ is obtained from 
the half of the great circle $(x_0,x_0^{\prime})$ by reflecting from the boundaries.}
 \end{figure}

If $\cot \left(\frac{L}{2}\right) \neq 0$, then we have the same contradiction as in the hyperbolic case. When $L=(2n+1)\pi$,  we have $\cot \left(\frac{L}{2}\right) =0$ and $k_g=0$. In this case,
 there could exist open sets of 3-period orbits. 

Now, we discuss
the characteristics of the billiards on which open sets of 3-period orbits exist. 
Note that only orbits without repetition are considered, and this assumption limits our
 cases to $L=\pi, 3\pi$, or $5\pi$, see also \cite{gutkin_t}.

\begin{proposition}
Consider a spherical triangle $x_0,x_1,x_2$ on the unit sphere with perimeter
 $L=\pi,3\pi$ or $5\pi$.  Let $l_0,l_1,l_2$ be the great circles passing through the vertices
 orthogonal 
to the corresponding bisectors. Then, these great circles intersect at the
right angles and any billiard boundary containing segments of $l_0,l_1,l_2$ passing through 
$x_0,x_1,x_2$ will have an open set of 3-period orbits.
 \end{proposition}

\begin{proof}
Let $l_0$ be a geodesic on $\mathbb{S}^2$ and $x_0$ be any point on $l_0$. Create two geodesics, 
$l_1$ and $l_2$, that are perpendicular to $l_0$ and to each other, but do not pass
 through $x_0$. Consider any geodesic segment, $\tau$, of length $\pi$, whose endpoint
 is $x_0$. Denote the angle between $\tau$ and $l_0$ as $\sigma$.  Through two reflections
 over $l_1$ and $l_2$, this line $\tau$ forms a triangle of length $\pi$ within the boundary
 created by $l_0$, $l_1$, and $l_2$. Since $x_0$ and $\sigma$ were arbitrary, any 3-period
 orbit of length $\pi$ must be contained in one octant, which is formed by $l_0$, $l_1$,
 and $l_2$, whose intersections are orthogonal. In particular, this implies that all
 orbits in the octant are 3-periodic except those which hit the corners. \\
Consider three great circles that intersect at $x_0$, $x_1$, and $x_2$. The total length of
 the lines is $6\pi$. This implies that an orbit of length $5\pi$ is the complement
 of $\triangle x_0x_1x_2$. It follows that an orbit of length $5\pi$ must have vertices
 on $l_0$, $l_1$, $l_2$ as in the $\pi$ case. \\
Now we consider a 3-period orbit of length $3\pi$. Since it is impossible to create
 an orbit where $\tau_0=\tau_1=\tau_2=\pi$, we look at the two other possible cases;
 when $0<\tau_0<\pi$, $\pi<\tau_1, \tau_2 <2\pi$, and $0<\tau_0, \tau_1<\pi$,
 $\pi<\tau_2 <2\pi$. In case 1, as shown in Figure \ref{fig:fig3}, we know that
 $\triangle x_0x_1x'_2$ has perimeter $\pi$ and that $x_2$ is antipodal to $x'_2$.
 This implies that $x_2$ lies on $l_2$. Note that case 2 is simply the complement of
 case 1. Therefore, we conclude that a 3-period orbit of length $3\pi$ also
 has vertices on $l_0$, $l_1$, and $l_2$. 
\end{proof}

\begin{figure}
\centering
    \includegraphics[width=1\textwidth]{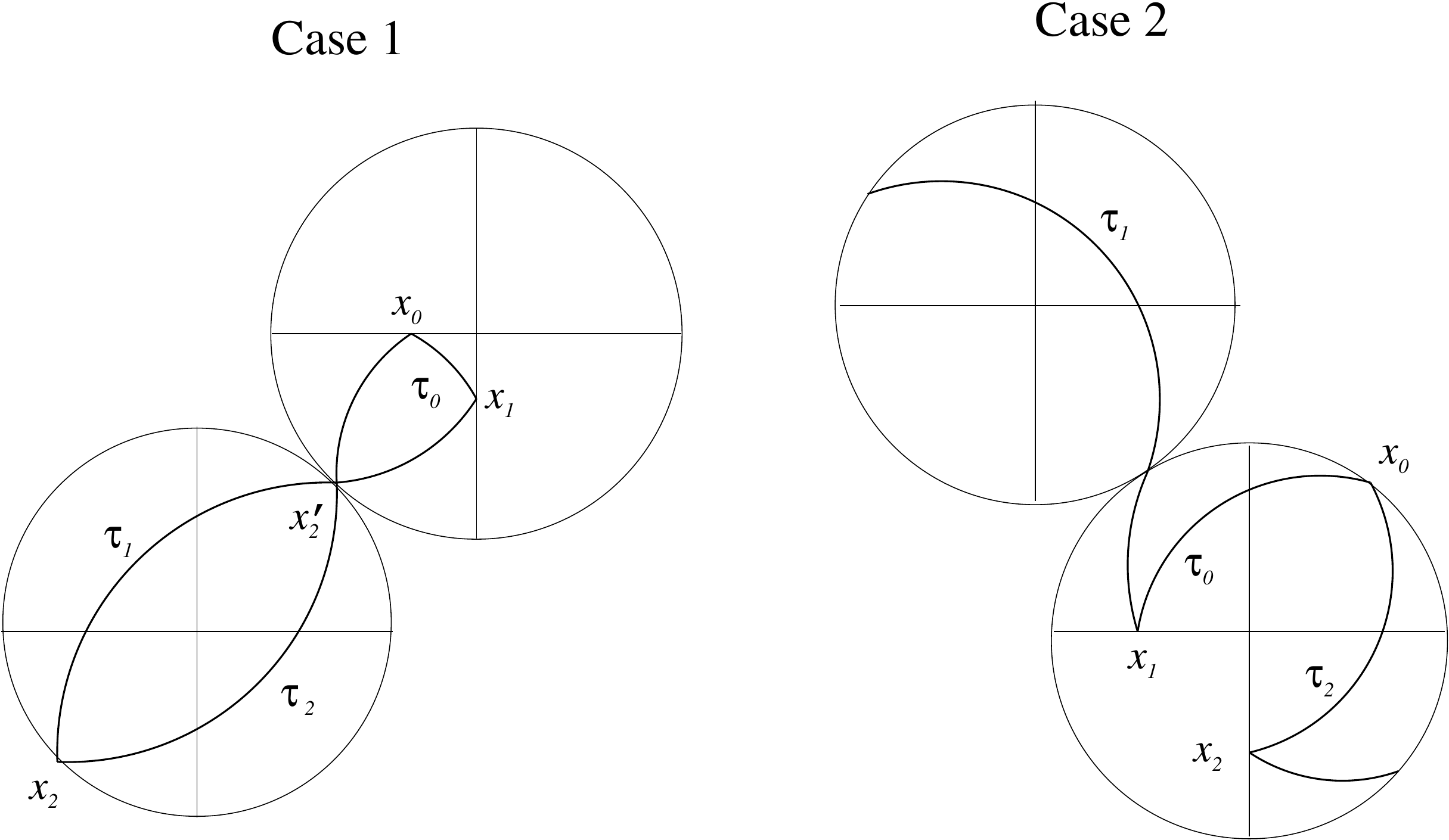}

\vspace{5mm}

\caption{The special cases $L=3\pi$ and $L=5\pi$.}
\label{fig:fig3}
 \end{figure}

The last proposition completely classifies the special cases when open sets of 3-period
orbits occur. If a given 3-period orbit has perimeter $L\neq \pi, 3\pi, 5\pi$ then the relation 
\eqref{eq:sphrel} implies that this orbit has an empty interior in $P_3$. If 
$L= \pi, 3\pi$ or $5\pi$ but for some vertex $x_i$ the geodesic curvature $k_g(s)$ 
does not vanish identically on any open boundary arc containing  $x_i$, 
then  \eqref{eq:sphrel} again leads to  the same contradiction. 

Finally applying the argument in \cite{wojtk}, we obtain that if the special cases do not occur 
the set of 3-period orbits has zero measure. This ends the proof of Theorem \ref{theo:sphere}.

\section{Acknowledgements}

We acknowledge support from National Science Foundation grant 
DMS 08-38434 ”EMSW21-MCTP: Research Experience for Graduate Students. 

We would also like to thank Y.M. Baryshnikov for providing us with a copy of his 
unpublished manuscript and for many useful suggestions. We also thank 
 the other participants of Applied Dynamics and Applied Analysis REGs at UIUC
for useful comments. Our special thanks go to N. Klamsakul for collaborating
 with us at the early stage of this project.

\end{document}